\documentclass[11pt]{amsart}

\usepackage[T1]{fontenc}
\usepackage[utf8]{inputenc}
\usepackage{fouriernc}       
\usepackage{microtype}

\usepackage[
  top    = 1.35in,
  bottom = 1.35in,
  inner  = 1.25in,
  outer  = 1.25in,
  footskip = 30pt
]{geometry}

\usepackage{titlesec}

\titleformat{\section}
  [block]
  {\normalfont\scshape\centering}
  {\S\,\Roman{section}.}
  {0.6em}
  {}
  [\vspace{3pt}{\titlerule[0.4pt]}\vspace{-2pt}]

\titlespacing*{\section}{0pt}{2.4ex plus .4ex}{1.6ex plus .2ex}

\titleformat{\subsection}
  [runin]
  {\normalfont\scshape}
  {\thesubsection.}
  {0.5em}
  {}
  [.\enspace]

\titlespacing*{\subsection}{0pt}{1.4ex plus .3ex}{0pt}

\renewcommand{\thesection}{\Roman{section}}
\renewcommand{\thesubsection}{\thesection.\arabic{subsection}}

\usepackage{amsmath, amssymb, amsfonts, amsthm, mathtools}
\numberwithin{equation}{section}
 \usepackage{booktabs}
 \usepackage{array}
 \usepackage{tabularx}

\newtheoremstyle{classical}
  {7pt}{5pt}{\itshape}{\parindent}{\scshape}{.}{ }{}

\newtheoremstyle{classicaldef}
  {7pt}{5pt}{\normalfont}{\parindent}{\scshape}{.}{ }{}

\theoremstyle{classical}
\newtheorem{theorem}{Theorem}[section]
\newtheorem{proposition}[theorem]{Proposition}
\newtheorem{lemma}[theorem]{Lemma}

\theoremstyle{classicaldef}
\newtheorem{definition}[theorem]{Definition}

\newtheorem{remark}[theorem]{Remark}

\usepackage{graphicx}
\usepackage{booktabs}
\usepackage{array}
\usepackage{subcaption}
\usepackage{tikz}
\usetikzlibrary{
  fit, backgrounds, arrows.meta,
  positioning, calc, decorations.pathreplacing
}
\usepackage{tikz-cd}
\usepackage{float}

\usepackage{enumitem}
\setlist{leftmargin=*, itemsep=2pt, topsep=4pt}

\usepackage[square]{natbib}         
\usepackage{hyperref}
\hypersetup{colorlinks,allcolors=black}
\usepackage[nameinlink, noabbrev]{cleveref}
\begin{document}

\title{On $ U(1)^{n-2} $-Invariant Special Lagrangian $ n $-Folds}
\author{By Mia S.L.\ Beard \\ \textit{University of Oxford}}
\date{\today}
\maketitle

\begin{abstract}

  This paper develops a construction of families of $ U(1)^{n-2} $-invariant special Lagrangian $ n $-folds in $ \mathbb{C}^{n} $, extending the analytic framework introduced by Joyce ($ n = 3 $) to arbitrary dimension. By reducing the special Lagrangian condition to a quasilinear elliptic system of two-dimensional non-linear Cauchy-Riemann equations, we analyse both the resulting geometry and its degenerations at singular points. We show that the structure and multiplicity of singularities are governed by an associated polynomial arising from the symmetry reduction. Explicit examples are constructed, including affine and perturbative solutions, and are compared with the classical Harvey-Lawson $ U(1)^{n-1} $-invariant submanifolds. We further show that the key elements of Joyce's analysis in the non-singular case, in particular the potential formulation and Dirichlet problem, extend to this higher-dimensional setting, with the proofs unchanged. 
  
\end{abstract}

\vspace{1.5em}

\section{Introduction}

There is well-established theory surrounding special Lagrangian $ 3 $-folds in $ \mathbb{C}^{3} $. Of particular relevance here is Joyce's paper \cite{joyce2005u1}, which establishes a construction for $ U(1) $-invariant special Lagrangian $ 3 $-folds in $ \mathbb{C}^{3} $, whereby a $ U(1) $-invariant submanifold $ N $, which can be written as a graph of functions $ u, v: \mathbb{R}^{2} \to \mathbb{R} $, is realised as a special Lagrangian $ 3 $-fold if and only if it satisfies a system of non-linear Cauchy-Riemann equations which depend on $ a $, where $ a $ is defined by $ |z_{1}|^{2} - |z_{2}|^{2} = 2a $. If $ a \neq 0 $, then $ u, v $ are smooth and $ N $ is a non-singular special Lagrangian $ 3 $-fold, but if $ a = 0 $, then there may be points of the form $ (x, y) $, where $ u, v $ are not differentiable, which corresponds to singular points of $ N $. In the non-singular case, the existence and uniqueness of a large class of non-singular $ U(1) $-invariant special Lagrangian $ 3 $-folds follows intuitively. In Joyce's further papers \cite{joyce2005u2, joyce2005u3}, it is recognised that in the case of $ a = 0 $, the Cauchy-Riemann equations are not always elliptic, leading to technical analytic issues. Proving existence and uniqueness in this case is proven using a priori estimates for derivatives of solutions, using these to show existence and uniqueness of weak solutions $ u, v $ to two Dirichlet problems when $ a = 0 $, which are continuous and weakly differentiable. In studying the singularities, it is found that, under mild conditions, the singularities are isolates, and have a multiplicity $ n > 0 $, and one of two types.

There is, however, comparatively little known about explicit symmetric constructions in higher-dimensions. The present paper addresses this gap by extending Joyce's framework to arbitrary dimension, considering the larger symmetry group $ U(1)^{n-2} $, recovering a rich geometry while retaining analytic tractability.

To orient the reader, we briefly describe the construction here.

Let $ \mathbb{C}^{n} $ be equipped with its standard symplectic structure. Motivated by Joyce's $ U(1) $-invariant construction of special Lagrangian submanifolds in $ \mathbb{C}^{3} $ \cite{joyce2005u1, joyce2005u2, joyce2005u3}, we consider the $ U(1)^{n-2} $-action on $ \mathbb{C}^{n} $, given by

\begin{equation}
    (e^{i \theta_{1}}, \dots, e^{i \theta_{n-2}}) \cdot (z_{1}, \dots, z_{n}) = (e^{i \theta_{1}}z_{1}, \dots, e^{i \theta_{n-2}}z_{n-2}, e^{-i(\theta_{1} + \cdots + \theta_{n-2})}z_{n-1}, z_{n}),
\end{equation}

so that the product $ z_{1} \cdots z_{n-1} $ is invariant.

Let $ S \subset \mathbb{R}^{2} $ be a domain with coordinates $ (x, y) $, and let $ u, v: S \rightarrow \mathbb{R} $ be functions in $ C^{1} $. Fix real constants $ a_{1}, \dots, a_{n-1} $, and define

\begin{equation*}
    P(w) = \prod_{k=1}^{n-1}(w+a_{k}).
\end{equation*}

For each $ (v, y) \in \mathbb{R}^{2} $, there exists a unique real solution $ w = w(v, y) $ of $ P(w) = v^{2} + y^{2} $ satisfying $ w \geq - \min \{a_{1}, \dots, a_{n-1} \} $. Using this data, we define a subset $ N \subset \mathbb{C}^{n} $ by fixing the moment-map level sets associated to the $ U(1)^{n-2} $-action,

\begin{equation*}
N = \Bigl\{ z \in \mathbb{C}^{n} :
|z_{j}|^{2} - |z_{n-1}|^{2} = a_{j} - a_{n-1}\ (j=1, \dots, n-2),\ 
i^{n-3} z_{1} \cdots  z_{n-1} = v+iy,\  z_{n} = x+iu \Bigr\}.
\end{equation*}

We show that $ N $ is a (possibly singular) $ U(1)^{n-2} $-invariant special Lagrangian $ n $-fold if and only if $ u $ and $ v $ satisfy 

\begin{equation*}
    u_{x} = v_{y}, \qquad v_{x} = -P'(w(v, y)) u_{y},
\end{equation*}

where $ P'(w) = \sum_{k=1}^{n-1} \prod_{i \neq k}(w + a_{i}) $.

If $ \min(a_{1}, \dots, a_{n-1}) $ is achieved by more than one of $ a_{1}, \dots, a_{n-1} $, then $ N $ is singular at points where $ v = y = 0 $. This is because the $ U(1)^{n-2} $-action is not free at such points, but the orbit collapses to $ U(1)^{k} $ for $ 0 \leq k < n-2 $.

The layout of this paper will be as follows. In Section~\ref{SECT:BackgroundOnSLManifolds} we give the essential background information on special Lagrangian submanifolds which will be needed, before introducing the symplectic and reduction geometry which underlies our construction in Section~\ref{SECT:SetUpOfGeometricConstruction}. The main theorem is precisely stated in Section~\ref{SEC:MainTheorem}, which we then prove in Section~\ref{SECT:ProofOfTheMainTheorem}. As a sanity check, we show that one can recover Joyce's $ n = 3 $ construction in Section~\ref{SECT:RecoveryOfJoycesEQS}, and give some examples in the form of Harvey-Lawson $ U(1)^{n-1} $-invariant solutions and translations in $ z_{n} $, and affine and perturbative special Lagrangians. Finally, we generalise Joyce's analysis of the reduced equations in Section~\ref{SECT:GeneralisingJoycesAnalysisOnTheReducedEquations} for the non-singular case, demonstrating that the proofs of Joyce on generating $ u, v $ from a potential and the Dirichlet problem is unchanged, commenting also on the winding number technique.

\section{Background on Special Lagrangian Submanifolds}\label{SECT:BackgroundOnSLManifolds}

We assume that the reader is familiar with the basic notions of calibrated geometry, and in particular, the framework introduced by Harvey and Lawson \cite{harvey_calibrated_1982}. In this section, we recall only those facts concerning special Lagrangian submanifolds that will be used later in this paper, fixing conventions and notation.

In $ \mathbb{C}^{n} $, special Lagrangian submanifolds are defined as follows.

\begin{definition}[Special Lagrangian manifolds in $ \mathbb{C}^{n} $]

Let $ (z_{1}, \dots, z_{n}) $ denote standard complex coordinates on $ \mathbb{C}^{n} $, equipped with the Riemannian metric $ g = |dz_{1}|^{2} + \cdots + |dz_{n}|^{2} $, the Kähler form $ \omega = \tfrac{i}{2}(dz_{1} \wedge d \bar{z}_{1} + \cdots + dz_{n} \wedge d \bar{z}_{n}) $, and the holomorphic volume form $ \Omega = dz_{1} \wedge \cdots \wedge dz_{n} $. Let $ L \subset \mathbb{C}^{n} $ be an oriented real submanifold of dimension $ n $. We say that $ L $ is a special Lagrangian submanifold of $ \mathbb{C}^{n} $ if it is calibrated with respect to $ \Re(\Omega) $. Equivalently, $ L $ admits an orientation making it into a special Lagrangian $ n $-fold if and only if $ \omega|_{L} \equiv 0 $ and $ \Im(\Omega)|_{L} \equiv 0 $ (see \cite[Corollary~III.11.]{harvey_calibrated_1982}).
    
\end{definition}

Special Lagrangian submanifolds have attracted significant attention following the work of \cite{harvey_calibrated_1982}. One of the topics which brought the importance of special Lagrangians into light was the Strominger-Yau-Zaslow (SYZ) conjecture \cite{strominger1996mirror}, which proposes that mirror symmetry for Calabi-Yau threefolds should be understood in terms of dual special Lagrangian torus fibrations. A central difficulty in this programme is that such fibrations are expected to have singular fibres, and understanding the structure of these singularities remains a key problem. 

In \cite{mclean1998deformations}, the deformations of calibrated submanifolds was studied, where McLean showed that the moduli space of special Lagrangian deformations of a compact special Lagrangian $ L $ is a smooth manifold of dimension equal to $ b^{1}(L) $ (the first Betti number).

 An influential set of papers were \cite{joyce2005u1, joyce2005u2, joyce2005u3}, whereby Joyce used symmetry-reduction techniques to construct large families of $ U(1) $-invariant special Lagrangian $ 3 $-folds in $ \mathbb{C}^{3} $. His studies included the singular case, where Joyce was able to prove existence and uniquness of continuous weak solutions, and understand the geometry of the singularities, as well as constructing $ U(1) $-invariant special Lagrangian fibrations of open sets in $ \mathbb{C}^{3} $. In a complementary direction, foundational research \cite{HaskinsCones} focused on the classification, existence, and geometric properties of special Lagrangian cones in $ \mathbb{C}^{m} $, the main result being the construction of an infinite family of special Lagrangian cones in $ \mathbb{C}^{3} $, each of which has a toroidal link.

We recall a characterisation of special Lagrangian $ n $-planes in $ \mathbb{C}^{n} $. Following Joyce \cite[Proposition 2.4.]{joyce2005u1}, this characterisation can be phrased in terms of a calibrated cross product associated to $ \Re(\Omega) $. To begin with, let us recall the definition of the calibrated cross product, which is the higher-dimensional analogue of the ordinary cross product in $ \mathbb{R}^{3} $, defined in the setting of calibrated geometry.

\begin{definition}[Cross Product on $ \mathbb{C}^{n} $]

Let $ (\mathbb{C}^{n}, g, \Re(\Omega)) $ be equipped with its standard flat metric and calibration. The cross product associated to $ \Re(\Omega) $ is the real multilinear map $ \times: (\mathbb{C}^{n})^{n-1} \rightarrow \mathbb{C}^{n} $, uniquely defined by the requirement that, for all vectors $ v_{1}, \dots, v_{n-1} \in \mathbb{C}^{n} $,

\begin{equation}\label{EQ:CrossProductOnComplexSpace}
    g(v_{1} \times \cdots \times v_{n-1}, w) = \Re(\Omega)(v_1,\dots,v_{n-1}, w).
\end{equation}

This defines a real tensor on $ \mathbb{R}^{2n} $.
    
\end{definition}

We offer a remark on how we will calculate this object.

\begin{remark}{Coordinate Expression of the Calibrated Cross Product}

    For explicit computations, it is convenient to describe the cross product in complex coordinates. Let $ \Phi := \mathrm{Re}(\Omega) $ denote the special Lagrangian calibration on $ \mathbb{C}^{n} $. Given vectors $ v_{1},\dots, v_{n-1} \in \mathbb{C}^{n} $, define the real $ 1 $-form

    \begin{equation*}
        W_{\Phi}(v_1,\dots,v_{n-1})(w) := \Phi(v_{1},\dots,v_{n-1}, w), \qquad w \in \mathbb{C}^{n},
    \end{equation*}

    which is an element of $ \Lambda^{1} T_{\mathbb{R}}^{*} \mathbb{C}^{n} $. We note that we have introduced the notation $ W_{\Phi} $ to make it clear that this is a generalised calibrated cross product.

    Under complexification, we have the decomposition

    \begin{equation*}
        (\Lambda^{1} T_{\mathbb{R}}^{*} \mathbb{C}^{n}) \otimes_{\mathbb{R}}\mathbb{C} = \langle dz_{1},\dots, dz_{n} \rangle_{\mathbb{C}} \;\oplus\; \langle d\bar{z}_{1},\dots,d\bar{z}_{n} \rangle_{\mathbb{C}}.
    \end{equation*}

    Accordingly, the contracted form may be written as

    \begin{equation*}
        W_{\Phi}(v_{1},\dots, v_{n-1}) =
\sum_{j=1}^{n} \bigl( a_{j} \,dz_{j} + \bar{a}_{j} \,d \bar{z}_j \bigr), \qquad a_{j} \in \mathbb{C}.
    \end{equation*}

    Contracting this $ 1 $-form with the metric $ g $ yields the corresponding vector

    \begin{equation*}
        v_{1} \times \cdots \times v_{n-1} = \sum_{j=1}^{n} \bigl(a_{j} \tfrac{\partial}{\partial \bar z_{j}} + \bar{a}_{j} \tfrac{\partial}{\partial z_{j}} \bigr).
    \end{equation*}

    In particular, the coefficients $ a_{j} $ may be computed explicitly as complex determinants arising from the contraction of $ \Omega $ with $ v_{1}, \dots, v_{n-1} $.
    
\end{remark}

With this definition, we may propose a higher-dimensional generalisation of \cite[Proposition~2.4]{joyce2005u1}.

\begin{proposition}[Extension of an $ (n-1) $-plane to a special Lagrangian plane]\label{PROP:JoyceSLMomentMap}

   Let $ (\mathbb{C}^{n}, g, \Re(\Omega)) $ be equipped with its standard flat metric and calibration. Let $ \vec{x} \in \mathbb{C}^{n} $, and let $ v_{1}, \dots, v_{n-1} \in T_{\vec{x}} \mathbb{C}^{n} $ be linearly independent vectors spanning an isotropic $ (n-1) $-dimensional subspace. Then there exists a unique special Lagrangian $ n $-plane $ P \subset T_{\vec{x}} \mathbb{C}^{n} $ containing $ \langle v_{1}, \dots, v_{n-1} \rangle $, and it is given by

   \begin{equation*}
       P = \langle v_{1}, \dots, v_{n-1}, v_{1} \times \cdots \times v_{n-1} \rangle.
   \end{equation*}
    
\end{proposition}

\section{Set-Up of the Geometric Construction}\label{SECT:SetUpOfGeometricConstruction}

We begin by introducing the symplectic and reduction geometry underlying the construction.

\subsection{The Symplectic Set-Up}\label{SUBSEC:SymplecticSetUp}

Let $ \mathbb{C}^{n} $ be equipped with its standard symplectic form $ \omega = \tfrac{i}{2}\sum_{j=1}^{n} dz_{j} \wedge d\bar{z}_{j} $. For $ (e^{i \theta_{1}}, \dots, e^{i \theta_{n-2}}) \in U(1)^{n-2} $ and $ (z_{1}, \dots, z_{n}) \in \mathbb{C}^{n} $, define the action

\begin{equation}\label{EQ:LieGroupAction}
    (e^{i \theta_{1}}, \dots, e^{i \theta_{n-2}}) \cdot (z_{1}, \dots, z_{n}) = (e^{i \theta_{1}}z_{1}, \dots, e^{i \theta_{n-2}}z_{n-2}, e^{-i(\theta_{1} + \dots + \theta_{n-2})}z_{n-1}, z_{n}).
\end{equation}

This action is Hamiltonian with respect to $ \omega $ and preserves the holomorphic volume form $ \Omega = dz_{1} \wedge \cdots \wedge dz_{n} $.

Associated to the action \eqref{EQ:LieGroupAction} is the standard moment map $ \mu: \mathbb{C}^{n} \rightarrow \mathbb{R}^{n-2} $, given explicitly by

\begin{equation}\label{EQ:MomentMap}
    \mu(z) = \bigl(|z_{1}|^{2} - |z_{n-1}|^{2}, \dots, |z_{n-2}|^{2} - |z_{n-1}|^{2} \bigr).
\end{equation}

Any connected $ U(1)^{n-2} $-invariant Lagrangian submanifold of
$ \mathbb{C}^{n} $ lies entirely within a single level set of the moment map \cite[Proposition 4.1.]{joyce_special_2002}. We therefore introduce coordinates adapted to these level sets.

\subsection{Invariant Coordinates and the Reduced Space}\label{SUBSEC:InvariantCoordsAndTheReducedSpace}

On the open subset $ \{ z_{1} \cdots z_{n-1} \neq 0 \} \subset \mathbb{C}^{n} $, we write that locally

\begin{equation}\label{EQ:PolarRewrite}
    z_{j} = r_{j} e^{i \theta_{j}} 
    \quad (1 \leq j \leq n-1), \qquad z_{n} = x + iu,
\end{equation}

with $ r_{j} > 0 $ and $ \theta_{j} \in \mathbb{R} $.

Fix parameters $ \vec{a} = (a_{1}, \dots, a_{n-1}) \in \mathbb{R}^{n-1} $, which determine a moment map value via

\begin{equation*}
    \mu = (a_{1} - a_{n-1}, \dots, a_{n-2} - a_{n-1}),
\end{equation*}

which is a choice made to preserve permutation symmetry among the coordinates $ z_{1}, \dots, z_{n-1} $. By definition of the moment map, points in $ \mu^{-1}(\vec{a}) $ satisfy

\begin{equation}\label{EQ:LevelSetRelations}
    |z_{j}|^{2} - |z_{n-1}|^{2} = a_{j} - a_{n-1} \qquad (j = 1, \dots, n-2),
\end{equation}

and the level set $ \mu^{-1}(\vec{a}) $ can be written as

\begin{equation}
    \mu^{-1}(\vec{a}) = \{z \in \mathbb{C}^{n}: |z_{j}|^{2} - |z_{n-1}|^{2} = a_{j} - a_{n-1} \}.
\end{equation}

These relations imply that the quantities $ r_{j}^{2} - a_{j} $ coincide for $ j = 1, \dots, n-1 $. We therefore define a real-valued function $ w $ on $ \mu^{-1}(\vec{a}) $, defined by

\begin{equation}\label{EQ:DefinitionOfw}
    w = r_{j}^{2} - a_{j} \quad (j = 1, \dots, n-1).
\end{equation}

The quotient space $ \mu^{-1}(\vec{a}) / U(1)^{n-2} $ is four-dimensional, and so we introduce coordinates $ (x, y, u, v) $ on this reduced space by

\begin{equation}\label{EQ:InvariantCoords}
    v + iy = i^{n-3}z_{1} \cdots z_{n-1}, \qquad x + iu = z_{n}.
\end{equation}

Writing out the product constraint explicitly yields

\begin{equation*}
    i^{3-n} z_{1} \cdots z_{n-1} = i^{3-n} \prod_{j=1}^{n-1} \sqrt{w+a_{j}} e^{i(\theta_{1} + \cdots + \theta_{n-1})},
\end{equation*}

and hence

\begin{equation*}
    P(w) := \prod_{j=1}^{n-1}(w + a_{j}) = v^{2} + y^{2}.
\end{equation*}

For each $ (v,y) \neq (0,0) $, the equation $ P(w)=v^{2}+y^{2} $ admits a unique solution $ w(v,y) $ satisfying $ w \geq -\min(a_{j}) $.
This solution depends smoothly on $ (v,y) $ away from the origin, where the degeneracy of the polynomial becomes relevant and will be addressed later.

\begin{figure}[H]

    \centering
    \includegraphics[width=0.6\linewidth]{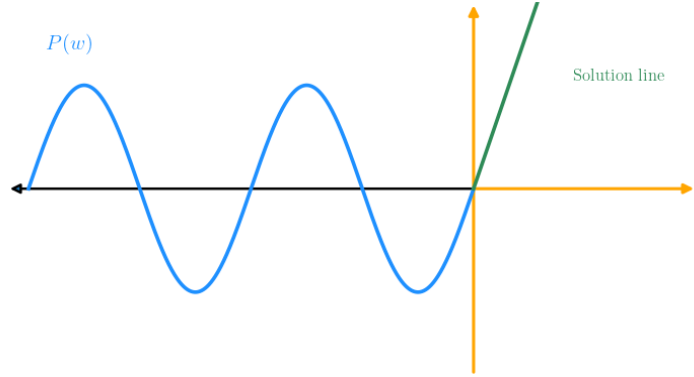}
    
    \caption{Visualisation of the relation $ P(w) = v^{2} + y^{2} $ and the distinguished solution branch.}
    
\end{figure}

\section{Statement of the Main Result}\label{SEC:MainTheorem}

We now formally state the main theorem of this paper.

\begin{theorem}[Construction of $ U(1)^{n-2} $-Invariant Special Lagrangian n-Folds]\label{THM:MainTheorem}

Fix parameters $ \vec{a} = (a_{1}, \dots, a_{n-1}) \in \mathbb{R}^{n-1} $, and let $ S \subset \mathbb{R}^{2} $ be an open domain with coordinates $ (x, y) \in S $.

Let $ P(w) = \prod_{j=1}^{n-1}(w+a_{j}) $.

\begin{itemize}

    \item Suppose that $ \min(a_{j}) $ has multiplicity one. Let $ u, v \in C^{1}(S) $ be real-valued functions. For $ (v, y)\neq (0, 0) \in \mathbb{R}^{2} $ let $ w = w(v, y) $ denote the unique real solution of $ P(w) = v^{2} + y^{2} $, satisfying $ w \geq - \min(a_{j}) $, as in Section~\ref{SUBSEC:InvariantCoordsAndTheReducedSpace}.

    Define a subset $ N \subset \mathbb{C}^{n} $ by

   \begin{align*}
    N=\Bigl\{ z \in \mathbb{C}^{n}:\;
    &|z_{j}|^{2}-|z_{n-1}|^{2}=a_{j}-a_{n-1}\quad (j=1,\dots,n-2),\\
    &i^{\,n-3}z_{1}\cdots z_{n-1}=v(x,y)+iy,\qquad
    z_{n}=x+iu(x,y)\Bigr\}.
    \end{align*}
    
    If $ u $ and $ v $ satisfy the non-linear first-order system

    \begin{equation}\label{EQ:MainEquations}
    \begin{cases}
    u_{x}=v_{y},\\[0.3em]
    v_{x}=-P'(w(v,y))\,u_{y},
    \end{cases}
    \end{equation}

    on $ S $, where $ 
    P'(w)=\sum_{k=1}^{n-1}\prod_{i\neq k}(w+a_{i}) $, then $ N $ is a smooth $ U(1)^{n-2} $-invariant special Lagrangian
$ n $-fold.
Moreover, $ N $ extends as a $ C^{1} $ submanifold up to $ \partial S $
with respect to the above parametrisation.

    \item If $ \min(a_{j}) $ has multiplicity greater than one, assume $ u, v \in C^{0}(S) $ and that \eqref{EQ:MainEquations} holds on $ S \setminus \{(x,y): (v(x,y),y)=(0,0)\} $. Then the subset $ N $ defined above is a $ U(1)^{n-2} $-invariant special
Lagrangian $ n $-fold, which is singular precisely at points where
$ v=y=0 $.
    
\end{itemize}

\end{theorem}

The proof of Theorem~\ref{THM:MainTheorem} is given in the subsequent section.

\section{Proof of the Main Theorem}\label{SECT:ProofOfTheMainTheorem}

We now prove Theorem~\ref{THM:MainTheorem}, introducing first the necessary objects.

Fix a non-singular point $ \vec{z} \in N $. To determine when $ T_{\vec{z}}N $ is special Lagrangian, we compute an explicit basis of tangent vectors spanning $ T_{\vec{z}}N $, considering vectors tangent to the $ U(1)^{n-2} $-orbits and those transverse to them, and evaluate the calibrated cross product on this basis.

Since the special Lagrangian condition is invariant under the $ U(1)^{n-2} $-action, we may, without loss of generality, choose a representative of the orbit of $ \vec{z} $ for which

\begin{equation}
    \theta_1=\cdots=\theta_{n-1}.
\end{equation}

This gauge choice significantly simplifies the ensuing computations.

\subsection{Tangent Vectors Along the Fibres}\label{SUBSECT:TangentVectorsAlongFibres}

For each $ i = 1, \dots, n-2 $, consider the one-parameter subgroup of the $ U(1)^{n-2} $-action \eqref{EQ:LieGroupAction}, given by

\begin{equation*}
    \gamma_{i}(\theta) = (z_{1}, \dots, e^{i\theta} z_{i}, \dots, z_{n-2},
   e^{-i\theta} z_{n-1}, z_{n}).
\end{equation*}

Differentiating at $ \theta = 0 $ yields tangent vectors along the group orbits,

\begin{equation}
    W_{\phi_i}
:= \left.\frac{d}{d\theta}\right|_{\theta=0} \gamma_i(\theta)
= (0, \dots, i z_i, \dots, 0, -i z_{n-1}, 0),
\end{equation}

which together span the tangent space to the $ U(1)^{n-2} $-orbit through $ \vec{z} $.

\subsection{Tangent Vectors Transverse to the Group Orbits}

Since the construction and special Lagrangian condition are $ U(1)^{n-2} $-invariant, we may choose a representative in each orbit with $ \theta_{1} = \cdots = \theta_{n-1} $.

Let $ W:U \subset \mathbb{R}^{2} \to \mathbb{C}^{n} $ denote the position vector at a point in the submanifold, given by

\begin{equation}\label{EQ:PositionVector}
     W(x, y) = \bigl( \sqrt{w(x, y) + a_{1}} \; e^{i \theta(x, y)}, \dots, \sqrt{w(x, y) + a_{n-1}}\; e^{i \theta(x, y)}, x+iu(x, y) \bigr),
\end{equation}

where $ w=w(v(x,y),y) $ and $ \theta=\theta(v(x,y),y) $.

Differentiating with respect to $ x $ and $ y $ yields the transverse tangent vectors

\begin{equation}
     W_{x} = \bigl( \frac{w_{x} e^{i \theta}}{2 \sqrt{w + a_{1}}} + i \theta_{x} \sqrt{w + a_{1}} e^{i \theta}, ..., \frac{w_{x}e^{i \theta}}{2 \sqrt{w + a_{n-1}}} + i \theta_{x} \sqrt{w + a_{n-1}} e^{i \theta}, 1 + iu_{x} \bigr),
\end{equation}

\begin{equation}
    W_{y} = \bigl( \frac{w_{y}e^{i \theta}}{2 \sqrt{w + a_{1}}} + i \theta_{y} \sqrt{w + a_{1}}e^{i \theta}, ..., \frac{w_{y}e^{i \theta}}{2 \sqrt{w + a_{n-1}}} + i \theta_{y} \sqrt{w + a_{n-1}}e^{i \theta}, iu_{y} \bigr).
\end{equation}

The derivatives of $ w $ and $ \theta $ with respect to $ (x, y) $ are obtained by implicit differentiation of the defining relations

\begin{equation}\label{EQ:ProductFormOne}
    e^{i(n-1) \theta} = i^{3-n} \frac{v+iy}{\sqrt{v^{2} + y^{2}}},
\end{equation}

\begin{equation}\label{EQ:ProductFormTwo}
    \prod_{j = 1}^{n-1} (w+a_{j}) = v^{2} + y^{2},
\end{equation}

which hold away from points where $ (v, y) = (0, 0) $. This yields

\begin{align}
\theta_{x} &= - \frac{v_{x} \, y}{(n-1)(v^{2}+y^{2})}, &
\theta_{y} &= \frac{v-yv_{y}}{(n-1)(v^{2}+y^{2})}, \label{eq:theta-derivs}\\
w_{x} &= \frac{2vv_{x}}{(v^{2}+y^{2})\, \sum_{k=1}^{n-1}\frac{1}{w+a_{k}}}, &
w_{y} &= \frac{2(vv_{y}+y)}{(v^{2}+y^{2})\, \sum_{k=1}^{n-1}\frac{1}{w+a_{k}}}.
\end{align}

\begin{proof}

    For the transverse vectors, differentiate \eqref{EQ:PositionVector} with respect to $ x $ and $ y $, treating $ w $ and $ \theta $ as implicit functions of $ (x, y) $. Differentiation of \eqref{EQ:ProductFormTwo} with respect to $ x $ gives

    \begin{equation*}
        P'(w)w_{x} = 2vv_{x},
    \end{equation*}

    where the formula for $ w_{x} $ then follows by division by $ P'(w) $. The formula for $ w_{y} $ follows similarly, with the additional term $ 2y $ arising from $ \partial_{y}(y^{2}) $. For the gauge derivatives, for $ \theta_{x} $ we use \eqref{EQ:ProductFormOne} and differentiate with respect to $ x $, which gives 

    \begin{equation*}
        i(n-1)\theta_{x}e^{i(n-1)\theta} = i^{3-n} \bigl( \frac{v_{x}}{\sqrt{v^{2} + y^{2}}} - \frac{(v+iy)vv_{x}}{(v^{2}+y^{2})^{\frac{3}{2}}} \bigr).
    \end{equation*}

    Substituting \eqref{EQ:ProductFormOne} back in, and dividing both sides,

    \begin{equation*}
        i(n-1)\theta_{x} = \frac{-v_{x}y}{(n-1)(v^{2}+y^{2})}.
    \end{equation*}

    The formula for $ \theta_{x} $ follows by division by $ i(n-1) $. The derivation of $ \theta_{y} $ follows similarly.
    
\end{proof}

\begin{remark}\label{REMARK:Isotropic}

By construction, $ W_{\phi_{1}}, \dots, W_{\phi_{n-2}}, W_{x} $ are linearly independent, and $ \omega(W_{\phi_{i}}, W_{\phi_{j}}) = 0 $ for $ i \neq j, \quad i, j = 1, \dots, n-2 $, and $ \omega(W_{\phi_{i}}, W_{x}) = 0 $ for $ i = 1, \dots, n-2 $. This gives us that $ \langle W_{\phi_{1}}, \dots, W_{\phi_{n-2}}, W_{x}, W_{\Phi}(W_{\phi_{1}}, \dots, W_{\phi_{n-2}}, W_{x}) \rangle $ is isotropic.

\end{remark}

\subsection{Calibrated Cross Product}

Let $ e_{1}, \dots, e_{n} $ be the standard complex basis of $ \mathbb{C}^{n} $. For each $ j = 1, \dots, n $, define $ M^{(j)} $ to be the $ n \times n $ complex matrix whose columns are $ W_{\phi_{1}}, \dots, W_{\phi_{n-2}}, W_{x}, e_{j} $. Recall from Section~\ref{SECT:BackgroundOnSLManifolds} that the calibrated cross product $ W_{\Phi} $ is defined componentwise by 

\begin{equation*}
    W_{\Phi^{j}} := \overline{\mathrm{Det}(M^{(j)})},
\end{equation*}

where $ \mathrm{Det} $ is the complex determinant.

\begin{lemma}{Components of $ W_{\Phi} $}\label{LEM:Phi_I_Components}

Set $ A_{k} := \sqrt{w + a_{k}} $ for $ k = 1, \dots, n-1 $. Then the components of $ W_{\Phi}\bigl(W_{\phi_{1}},\dots,W_{\phi_{n-2}},W_{x}\bigr) $ are given as follows.

For $ 1 \leq i \leq n-1 $,

\begin{equation}
    (W_{\Phi})^{i} =
-(-i)^{n-1}e^{-i(n-2)\theta}
\left(\prod_{\substack{k=1\\ k\neq i}}^{n-1}\sqrt{w+a_{k}}\right)
(u_{x}+i)
\end{equation}

For $ i = n $,

\begin{equation}
  (W_{\Phi})^{n} = -iv_{x}.
\end{equation}

\end{lemma}

\begin{proof}

    Fix $ 1 \leq i \leq n-1 $, and let $ M^{(j)} $ denote the $ n \times n $ matrix, whose columns are $ W_{\phi_{1}}, \dots, W_{\phi_{n-2}}, W_{x}, e_{i} $. Expanding the determinant along the final column yields

    \begin{equation*}
        \mathrm{Det}(M^{(j)}) = (-1)^{j + n} \mathrm{Det}(\widetilde{M}^{(j)}),
    \end{equation*}

    where $ \widetilde{M}^{(j)} $ is the $ (n-1) \times (n-1) $ matrix obtained by deleting the $ i $-th row, and the final column.

    The row corresponding to the $ z_{n} $ coordinate contains a single non-zero entry of $ 1 + iu_{x} $ in the $ W_{x} $-column. Expanding along the row gives

    \begin{equation*}
        \mathrm{Det}(\widetilde{M}^{(i)}) = i^{n-3}(1 + iu_{x}) \mathrm{Det}(B^{(i)}),
    \end{equation*}

    where $ B^{(i)} $ is the $ (n-2) \times (n-2) $-matrix obtained by deleting the $ z_{n} $-row and the $ W_{x} $-column.

    The matrix $ B^{(i)} $ is made up of rows $ z_{1}, \dots, \hat{z_{i}},  \dots, z_{n-1} $, and columns $ W_{\phi_{1}}, \dots, W_{\phi_{n-2}} $. Each column $ W_{\phi_{r}} $ has precisely two non-zero entries of $ (W_{\phi_{r}})_{z_{r}} = iA_{r}e^{i \theta} $ and $ (W_{\phi_{r}})_{z_{n-1}} = -iA_{n-1}e^{i \theta} $.

    If $ i \neq n-1 $, the $ z_{n-1} $-row of $ B^{(i)} $ is non-zero. Expanding along this row, all cofactors vanish except the one corresponding to the column $ W_{\phi_{i}} $. The resulting minor is diagonal, with diagonal entries $ iA_{r} e^{i \theta} $ for $ r \in \{1,\dots,n-2\} \setminus \{i\} $. Hence,

    \begin{equation*}
        \mathrm{Det}(B^{(i)}) = i^{\,2n-5}e^{i(n-2)\theta} \prod_{\substack{k=1\\ k\neq i}}^{n-1}A_{k}.
    \end{equation*}

    Rewriting the $ i $-power and taking the complex conjugate yields the required equation.

The case $ i=n-1 $ is analogous: in this case $ B^{(n-1)} $ is diagonal with entries
$ iA_{r} e^{i\theta} $ for  $r=1,\dots,n-2 $, yielding the same expression.

For $ i = n $, we use cofactor expansion along the final column, to give

\begin{equation*}
    \det(M^{(n)}) = (-1)^{n+n}\det(\widetilde{M}^{(n)}) = \det(\widetilde{M}^{(n)}),
\end{equation*}

where $ \widetilde{M}^{(n)} $ is the $ (n-1) \times (n-1) $ matrix obtained by deleting the $ n $-th row and the final column. This differs from the previous computations as it now deletes the $ 1+iu_{x} $ term. The upper-left block is then diagonal, and so with one more iteration of the cofactor expansion gives 

\begin{equation*}
    \overline{(W_{\Phi})^{n}} = i^{n-2}i^{3-n}(v+iy)\bigl( \frac{w_{x}}{2} \sum_{k=1}^{n-1} \frac{1}{2+a_{k}} + i(n-1)\theta_{x} \bigr)
\end{equation*}

This expression can be simplified by noting that

\begin{equation*}
    \frac{w_{x}}{2} \sum_{k} \frac{1}{w+a_{k}} + i(n-1)\theta_{x} = \frac{v_{x}}{v^{2}+y^{2}},
\end{equation*}

and so

\begin{equation*}
    \overline{(W_{\Phi})^{n}} = iv_{x}.
\end{equation*}

Thus, taking the complex conjugate gives the claimed form for $ (W_{\Phi})^{(n)} $.

\end{proof}

\subsection{Proof of Theorem~\ref{THM:MainTheorem}}

We now have the necessary components to proceed with the proof of Theorem~\ref{THM:MainTheorem}.

\begin{proof}[Proof of Theorem~\ref{THM:MainTheorem}]

Let $ S \subset \mathbb{R}^{2} $ be a domain with coordinates $ (x,y) $, and let
$ u,v : S \to \mathbb{R} $ be $ C^{1} $-functions. At a non-singular point, the tangent space of $ N $ is spanned by

\begin{equation*}
    W_{\phi_{1}}, \dots, W_{\phi_{n-2}}, W_{x}, W_{y},
\end{equation*}

where the vectors $ W_{\phi_{i}} $ span the $ U(1)^{n-2} $-orbit and
$ W_{x}, W_{y} $ are transverse directions. By Remark~\ref{REMARK:Isotropic}, the $ (n-1) $-plane $ \langle W_{\phi_{1}}, \dots, W_{\phi_{n-2}}, W_x \rangle $
is isotropic.  Denote by $ W_{\Phi} $ its calibrated cross-product vector, as defined in Section~\ref{SECT:BackgroundOnSLManifolds}.

By Proposition~\ref{PROP:JoyceSLMomentMap}, the $ n $-plane $ \langle W_{\phi_{1}}, \dots, W_{\phi_{n-2}}, W_{x}, W_{y} \rangle $ is special Lagrangian if and only if $ W_{y} $ lies in the span of $ \langle W_{\phi_{1}}, \dots, W_{\phi_{n-2}}, W_{x}, W_{\bar{\Phi}} \rangle $. Equivalently, there exist real-numbers $ \alpha_{i}, \beta, \gamma $ such that

\begin{equation}\label{EQ:DecompOfW_Y}
    W_{y} = \sum_{i = 1}^{n-2} \alpha_{i} W_{\phi_{i}} + \beta W_{x} + \gamma W_{\bar{\Phi}}.
\end{equation}

To begin with, consider the $ n $-th component of \eqref{EQ:DecompOfW_Y}. Since the $ W_{\phi_{i}} $ have zero $ z_{n} $-components, we have that

\begin{equation}\label{EQ:nThComponents}
    iu_{y} = \beta(1+iu_{x}) + \gamma(-iv_{x}).
\end{equation}

Taking the real parts of both sides of \eqref{EQ:nThComponents}, we determine that $ \beta = 0 $. Considering the imaginary parts of \eqref{EQ:nThComponents}, one has

\begin{equation}\label{EQ:uyvx}
    u_{y} = -\gamma v_{x}.
\end{equation}

The $ i $-th component of \eqref{EQ:DecompOfW_Y}, for $ 1 \leq i \leq n-2 $, is

\begin{equation}\label{EQ:iTHComponentOne}
    \frac{w_{y}e^{i \theta}}{2\sqrt{w+a_{i}}} + i\theta_{y}\sqrt{w+a_{i}}e^{i \theta} = i\alpha_{i}\sqrt{w+a_{i}}e^{i \theta} + \gamma \bigl( -(-i)^{n-1} \bar{z}_{1} \cdots \hat{\bar{z}_{i}} \cdots \bar{z}_{n-1}(u_{x}+i) \bigr).
\end{equation}

Multiply both sides of \eqref{EQ:iTHComponentOne} by $ \sqrt{w+a_{i}} e^{-i\theta} $,

\begin{equation}\label{EQ:iTHComponentTwo}
    \frac{w_{y}}{2} + i\theta_{y}(w+a_{i})= i\alpha_{i}(w+a_{i}) + \gamma \bigl( -(-i)^{n-1} \bar{z}_{1} \cdots \bar{z}_{n-1}(u_{x}+i) \bigr).
\end{equation}

Since $ \bar{z}_{1} \cdots \bar{z}_{n-1} = (-i)^{3-n}(v-iy) $, we get that

\begin{equation}\label{EQ:iTHComponentThree}
    \frac{w_{y}}{2} + i\theta_{y}(w+a_{i})= i\alpha_{i}(w+a_{i}) + \gamma (v-iy)(u_{x}+i).
\end{equation}

Taking the real parts of both sides, 

\begin{equation}\label{EQ:iTHComponentRealParts}
   \frac{vv_{y}+y}{P'(w)} = \gamma(vu_{x}+y).
\end{equation}

Taking the imaginary parts of both sides,

\begin{equation}\label{EQ:iTHComponentImaginaryParts}
    \alpha_{i} = \theta_{y} - \frac{\gamma(v-yu_{x})}{w+a_{i}}
\end{equation}

Now looking at the $ (n-1) $-st component of \eqref{EQ:DecompOfW_Y}, which differs from \eqref{EQ:iTHComponentOne} in that now $ z_{n-1} $ contributes across all $ W_{\phi_{i}} $, we have that

\begin{equation}\label{EQ:n-1STComponent}
    \frac{w_{y}}{2} + i\theta_{y}(w+a_{n-1}) = -i \bigl( \sum_{k=1}^{n-2} \alpha_{k} \bigr)(w+a_{n-1}) + \gamma(v-iy)(u_{x}+i).
\end{equation}

Substituting in \eqref{EQ:iTHComponentImaginaryParts} to \eqref{EQ:n-1STComponent},

\begin{equation*}
    (n-1)\theta_{y} = \gamma(v-yu_{x})\sum_{k=1}^{n-1} \frac{1}{w+a_{k}}.
\end{equation*}

Then,

\begin{equation*}
    v-yv_{y} = \gamma P'(w)(v-yu_{x})
\end{equation*}

Together with the previous \eqref{EQ:iTHComponentRealParts}, we have, fundamentally, a linear algebra problem. The two equations in consideration are $ vv_{y}+y = \gamma P'(w)(vu_{x}+y) $ and $ v-yv_{y} = \gamma P'(w)(v-yu_{x}) $. Therefore, $ v(v_{y} - \gamma P'(w)u_{x}) + y(1-\gamma P'(w)) = 0 $ and $ -y(v_{y} - \gamma P'(w)u_{x}) + v(1- \gamma P'(w)) = 0 $. Solving this determines that

\begin{equation}
    \gamma = \frac{1}{P'(w)}, u_{x} = v_{y}.
\end{equation}

Substituting this expression for $ \gamma $ back into \eqref{EQ:nThComponents} determines the PDE system:

\begin{equation*}
    u_{x} = v_{y}, \qquad v_{x} = -P'(w) u_{y}.
\end{equation*}

Conversely, suppose that $ u_{x} = v_{y} $ and $ v_{x} = -P'(w)u_{y} $. Set

\begin{equation*}
    \beta =  0, \quad \gamma = \frac{1}{P'(w)}, \quad \alpha_{i} = \theta_{y} - \frac{v-yu_{x}}{P'(w)(w+a_{i})}, \quad 1 \leq i \leq n-2.
\end{equation*}

The the component equations are satisfied, and hence \eqref{EQ:DecompOfW_Y} holds. Therefore, the tangent $ n $-plane is special Lagrangian if and only if $ u_{x} = v_{y} $ and $ v_{x} = - P'(w)u_{y} $.

\end{proof}

\section{Recovery of Joyce's $ n = 3 $ Construction}\label{SECT:RecoveryOfJoycesEQS}

We begin by comparing Theorem \ref{THM:MainTheorem} with the $ U(1) $-invariant special Lagrangian $ 3 $-folds studied by Joyce \cite{joyce2005u1, joyce2005u2, joyce2005u3}, to prove that we can recover the equations used. Setting $ n = 3 $ in Theorem~\ref{THM:MainTheorem}, the polynomial 

\begin{equation*}
    P(w) = (w+a_{1})(w+a_{2})
\end{equation*}

satisfies

\begin{equation*}
    P'(w) = (w+a_{1}) + (w + a_{2}),
\end{equation*}

and hence

\begin{equation*}
    \gamma = - \frac{1}{(w + a_{1}) + (w + a_{2})}.
\end{equation*}

The moment map level is given by $ \mu = a_{1} - a_{2} $. Choosing the normalisation
$ a_{1} + a_{2} = 0 $, and writing $ a = a_{1} = -a_{2} $, the reduced equation $ P(w)=v^{2} + y^{2} $ becomes

\begin{equation*}
    w = \sqrt{v^{2} + y^{2} + a^{2}}.
\end{equation*}

Substituting into the system of Theorem~\ref{THM:MainTheorem} yields

\begin{equation}\label{EQ:JoyceU1System}
u_{x} = v_{y}, \qquad
v_{x} = -2\sqrt{v^{2} + y^{2} + a^{2}}\,u_{y}.
\end{equation}

This is precisely the non-linear Cauchy-Riemann system
studied by Joyce in the $ U(1) $-invariant case.

\section{Examples}\label{SECT:Examples}

\subsection{Harvey-Lawson $ U(1)^{n-1} $-Invariant Solutions and Translations in $ z_{n} $}

The construction of $ U(1)^{n-1} $-invariant special Lagrangian $ n $-folds in $ \mathbb{C}^{n} $ by \cite{harvey_calibrated_1982} provides a family of examples.

\begin{theorem}

    Let $ M_{b, c} $ denote the locus in $ \mathbb{C}^{n} $ of the equations

    \begin{equation}\label{EQ:HarveyLawsonMomentMaps}
        |z_{j}|^{2} - |z_{n}|^{2} = b_{j}, \quad j = 1, \dots, n-1.
    \end{equation}

    and
    \begin{equation}\label{EQ:HarveyLawsonRealImaginaryEquation}
        \begin{cases}
            \mathrm{Re}(z_{1} \cdots z_{n}) = c, \quad \text{if } n \text{ is even}, \\
            \mathrm{Im}(z_{1} \cdots z_{n}) = c, \quad \text{if } n \text{ is odd}.
        \end{cases}
    \end{equation}

    The smooth locus of $ M_{c} $ is special Lagrangian. When the defining equations have full rank, $ M_{c} $ is a smooth special Lagrangian submanifold.
     
\end{theorem}

Comparing the Harvey-Lawson moment maps with those of Theorem~~\ref{THM:MainTheorem}, and writing $ |z_{n}|^{2} = x^{2} + u^{2} $, we determine that

\begin{equation}\label{EQ:HL_gauge}
    w + a_{j} = x^{2} + u^{2} + b_{j}, \quad j = 1, \dots, n-1.
\end{equation}

Rearranging, $ w = x^{2}+u^{2}+b_{j}-a_{j} $. Then

\begin{equation}
    P(w) = \prod_{j=1}^{n-1}(w+a_{j}) = \prod_{j=1}^{n-1}(x^{2}+u^{2}+b_{j})
\end{equation}

Equation \eqref{EQ:HarveyLawsonRealImaginaryEquation} can be converted into our language. Writing that $ z_{1} \cdots z_{n} = i^{3-n}((vx-yu) + i(vu+xy)) $. 

If $ n $ is even, then $ z_{1} \cdots z_{n} = (-1)^{\frac{n}{2}+1}i((vx-yu)+i(vu+xy)) $. Then,

\begin{equation*}
    \mathrm{Re}(z_{1} \cdots z_{n}) = (-1)^{\frac{n}{2}}(vu+xy).
\end{equation*}

If $ n $ is odd, then $ z_{1} \cdots z_{n} = (-1)^{\frac{(n-3)}{2}}((vx-yu)+i(vu+xy)) $. Then 

\begin{equation*}
    \mathrm{Im}(z_{1} \cdots z_{n}) = (-1)^{\frac{(n-3)}{2}}(vu+xy).
\end{equation*}

In either case, since $ c $ is arbitrary, we can absorb the sign and write

\begin{equation*}
    uv+xy = c.
\end{equation*}

The system of two equations that we have determined is then

\begin{equation}
    v^{2}+y^{2} = \prod_{j=1}^{n-1}(x^{2}+u^{2}+b_{j}-a_{j}), \quad xy + uv = c.
\end{equation}

As for the question of solving this, one can rearrange for $ v $, that is, $ v = \frac{c-xy}{u} $, and substitute this into $ v^{2}+y^{2} = \prod_{j=1}^{n-1}(x^{2}+u^{2}+b_{j}-a_{j}) $. For example, in dimension $ 4 $, this gives a quartic in $ u^{2} $. Since quartics are soluble, there exists an explicit solution.

\subsection{Affine and Perturbative Special Lagrangians}

We now describe a particularly simple class of explicit solutions of \eqref{EQ:MainEquations} and explain their geometric meaning. These examples yield concrete special Lagrangian
$ n $-folds and provide natural base points for perturbative deformation analysis.

\begin{lemma}[Affine solutions]\label{LEM:AffineSolutions}

Let $ \alpha, \beta, \gamma \in \mathbb{R} $. The functions 

\begin{equation*}
    u(x,y) = \alpha x + \beta,\qquad v(x,y)=\alpha y+\gamma 
\end{equation*}

satisfy the system \eqref{EQ:MainEquations}.

\end{lemma}

Geometrically, the corresponding special Lagrangian submanifold is given by

\begin{equation*}
    z_1\cdots z_{n-1}=v+iy=\alpha y+\gamma+iy,
\qquad
z_n=x+i(\alpha x+\beta).
\end{equation*}

Writing

\begin{equation*}
    z_n=(1+i\alpha)x+i\beta=e^{i\theta}x+i\beta,
\qquad
e^{i\theta}=\frac{1+i\alpha}{\sqrt{1+\alpha^{2}}},
\end{equation*}

we see that

\begin{equation*}
    \Im\!\left(e^{-i\theta}z_{n}\right)=-\sin\theta\,\beta
\end{equation*}

is constant.

Similarly, multiplying the invariant monomial by $ e^{i\theta} $ yields

\begin{equation*}
    \Im\!\left(e^{i\theta}z_{1}\cdots z_{n-1}\right)=\text{constant},
\end{equation*}

so the phase of $ z_{1} \cdots z_{n-1} $ is fixed.

It follows that the resulting special Lagrangian $ n $-fold splits as a product

\begin{align}
    N
&=
\Bigl\{(z_1,\dots,z_{n-1})\in\mathbb{C}^{n-1}:
|z_i|^2-|z_{n-1}|^2=a_i-a_{n-1},\
\Im\!\left(e^{i\theta}z_1\cdots z_{n-1}\right)=\text{constant}
\Bigr\}
\\
&\quad\times
\Bigl\{z_n\in\mathbb{C}:
\Im\!\left(e^{-i\theta}z_n\right)=\text{constant}
\Bigr\}.
\end{align}

The first factor is a $ U(1)^{n-2} $-invariant special Lagrangian
$ (n-1) $-fold in $ \mathbb{C}^{n-1} $ of phase $ e^{-i\theta} $.
The second factor is a straight line in $ \mathbb{C} $,
a special Lagrangian $ 1 $-fold of phase $ e^{i\theta} $.
Their product therefore has phase $ 1 $ and is special Lagrangian in
$ \mathbb{C}^{n} $.

These affine solutions represent the simplest product-type special
Lagrangians arising from our framework and provide natural base points
for perturbative analysis.

\section{Generalising Joyce's Analysis of the Reduced Equations}\label{SECT:GeneralisingJoycesAnalysisOnTheReducedEquations}

In this section, we explain how the analytic
framework developed in \cite{joyce2005u1, joyce2005u2, joyce2005u3} extends to the present higher-dimensional setting.

\subsection{Generating $ u, v $ From a Potential and the Dirichlet Problem}

In \cite{joyce2005u1, joyce2005u2, joyce2005u3}, Joyce shows that solutions $ u, v \in C^{1}(S) $ of \eqref{EQ:MainEquations} may be obtained from a scalar potential
$ f \in C^{2}(S) $ satisfying a second-order quasilinear elliptic equation. We now generalise this formulation to the present setting.

If $ (u, v) $ satisfy \eqref{EQ:MainEquations}, then the condition $ u_{x} = v_{y} $ implies the local existence of a potential function $ f $ such that
$ f_{y} = u $ and $ f_{x} = v $. This leads to the following formulation.

\begin{proposition}[Potential formulation]

Let $ S \subset \mathbb{R}^{2} $ be a simply-connected domain, and let $ u, v \in C^{1}(S) $ satisfy \eqref{EQ:MainEquations}. Then there exists $ f \in C^{2}(S) $, unique up to addition of a constant, such that

\begin{equation*}
    f_{x} = v, \qquad f_{y} = u.
\end{equation*}

Moreover, $ f $ satisfies the second-order equation

\begin{equation}\label{EQ:SecondOrderF}
    f_{xx} + P'(w)f_{yy} = 0,
\end{equation}

where $ w=w(f_{x}, y) $ is the unique real solution of $ P(w)=f_x^{2} + y^{2} $, satisfying $ w \geq - \min(a_{j}) $.

Conversely, any function $ f \in C^{2}(S) $ solving \eqref{EQ:SecondOrderF} yields a solution $ (u, v) = (f_{y}, f_{x}) $ of \eqref{EQ:MainEquations}.

\end{proposition}

Equation \eqref{EQ:SecondOrderF} is a second-order quasilinear elliptic equation
wherever $ P'(w)>0 $, and becomes degenerate precisely along the locus
$ P'(w)=0 $. This naturally leads to the study of the Dirichlet problem for
\eqref{EQ:SecondOrderF}. In the following, we establish existence and uniqueness
results for this problem, generalising Theorem~6.2 of
\cite{joyce2005u2}.

\begin{theorem}[Existence and uniqueness of the potential function $ f $]\label{THM:DirichletProblem}

Let $ S \subset \mathbb{R}^{2} $ be a domain. Fix $ k \geq 0 $ and $ \alpha \in (0, 1) $, and let $ \phi \in C^{k+3, \alpha}(\partial S) $. Fix parameters $ \vec{a} = (a_{1}, \dots, a_{n-1}) \in \mathbb{R}^{n-1} $ and define $ P(w) = \prod_{j = 1}^{n-1} (w + a_{j}) $.

If $ \min(a_{j}) $ has multiplicity one, then there exists a unique solution $ f \in C^{k+3, \alpha}(S) $ of the Dirichlet problem

\begin{equation*}
     \begin{cases}
f_{xx}+P'(w)\,f_{yy}=0 & \text{in } S,\\
f=\phi & \text{on } \partial S,
\end{cases}
    \end{equation*}

where $ w=w(f_{x},y) $ is determined implicitly by $ P(w)=f_{x}^{2} + y^{2} $. The associated functions $ u = f_{y} $ and $ v = f_{x} $ lie in $ C^{k+2,\alpha}(S) $
and satisfy \eqref{EQ:MainEquations}.

\end{theorem}

Joyce's governing PDE system for $ U(1) $-invariant special Lagrangian submanifolds takes the form

\begin{equation*}
    u_{x} = v_{y}, \qquad v_{x} = -F(v^{2} + y^{2})u_{y},
\end{equation*}

with $ F(s) = 2 \sqrt{s + a^{2}} $. The associated potential equation is

\begin{equation*}
    f_{xx} + F(f_{x}^{2} + y^{2})f_{yy} = 0,
\end{equation*}

where the function $ F $ satisfies monotonicity, positivity, and regularity
assumptions to ensure ellipticity and compactness of solutions.

 In the present setting, the reduced equations take the same form, with $ F(s) = P'(w) $, where $ w = w(s) $ is determined implicitly by $ P(w) = s $. In the following, we verify that $ P'(w) $ satisfies the same structural
assumptions, and hence Joyce’s analytic results extend, without additional proof needed.

Consider the non-singular regime, in which $ \min(a_{j}) $ has multiplicity one. Let $ w_{0} = - \min(a_{j}) $, the smallest real root of $ P $, and define $ F(s) := P'(w(s)) $, where $ s \geq 0 $, and where $ w(s) $ is the unique real solution of $ P(w) = s $ satisfying $ w \geq w_{0} $.

     \begin{itemize}
     
         \item[(I.)] Since $ P $ is a polynomial, $ P' $ is smooth on $ [w_{0}, \infty) $. As $ P'(w)>0 $ on this interval, the inverse function theorem implies that
        $ w(s) $ depends smoothly on $ s $, and hence $ F $ is smooth on $ [0,\infty) $.

         \item[(II.)] Since $ w_{0} $ is a simple root of $ P $, we have $ P'(w_{0}) > 0 $. As $ P' $ has no zeroes on $ (w_{0}, \infty) $, it follows that $ P'(w) > 0 $ for all $ w > w_{0} $. Thus, $ F(s) \geq P'(w_{0}) > 0 $ for all $ s \geq 0 $, and the equation \eqref{EQ:SecondOrderF} is elliptic, with $ P' $ bounded from above in the domain.
         
         \item[(III.)] As $ P'(w) $ is positive on $ (w_{0},\infty) $ and $ w(s) $ is strictly increasing in
        $ s $, it follows that $ F(s) $ is increasing in $ s $.
        Equivalently, the coefficient $ A(s)=1/F(s) $ is decreasing, which is the
        monotonicity condition used in Joyce’s comparison and maximum principle
    arguments.
         
     \end{itemize}

     Therefore, $ F $ satisfies all the structural hypotheses required in Joyce’s
    analysis of the Dirichlet problem in the non-singular case, and the existence
    and uniqueness theorem follows by the same elliptic arguments as in
    \cite{joyce2005u1, joyce2005u2, joyce2005u3}.

     We expect that Joyce's analytic results in the singular case should extend to the current setting, however we do not seek proof of this here, as the required analysis is substantial \cite{joyce2005u2, joyce2005u3}. As for some brief remarks, Joyce treats the singular regime by passing to the potential formulation $ f_{x} = v $, $ f_{y} = u $, so that the system becomes a quasilinear equation of the form $ f_{xx}+F(f_x^{2}+y^{2})\,f_{yy}=0 $, where $ F $ is positive and may degenerate at the singular locus. The heart of the argument
is to develop a priori estimates for $ f $ (and hence for $ (u,v) $) which are uniform under
appropriate approximations, and then to solve the Dirichlet problem by compactness and
barrier methods.  In particular, Joyce imposes geometric hypotheses on the domain (notably
a symmetry under $ (x,y) \mapsto (x,-y) $ together with strict convexity conditions) which
enable the construction of barriers and the application of maximum-principle and winding
number arguments near the singular set. 

In our case, the potential equation has the same structure with
$ F(s)=P'(w(s)) $, where $ w(s) $ is defined implicitly by $ P(w)=s $ on the distinguished real
branch $ w \ge - \min(a_{j}) $.  We therefore expect that, under the same type of domain
hypotheses and with suitable control of the degeneracy locus $ P'(w)=0 $, Joyce's singular
Dirichlet theory and regularity results should carry over with only notational changes.
We leave a detailed verification to those who are interested.

\subsection{Winding Number Techniques}

We will say a few works on the winding number technique which Joyce describes \cite{joyce2005u1, joyce2005u2, joyce2005u3}. Throughout this subsection, let $ S \subset \mathbb{R}^{2} $ be a bounded domain, and let $ (u_{1}, v_{1}) $ and $ (u_{2}, v_{2}) $ denote two solutions of the reduced system
\eqref{EQ:MainEquations} on $ S $.  When $ P'(w)>0 $ on $ S $, we assume $ (u_{j}, v_{j}) \in C^{1}(S) $ for $ j=1,2 $, and when $ P'(w)=0 $
we allow $ (u_{j}, v_{j}) $ to be singular solutions.

Define the difference map

\begin{equation}\label{EQ:HolomorphicDifferenceMap}
F = (u_{1} - u_{2},\;v_{1} - v_{2}): S \rightarrow \mathbb{R}^{2} .
\end{equation}

We study the zeros of $ F $ via winding number, following Joyce.

Let $ C $ be a compact oriented $ 1 $-manifold, and $ \gamma: C \rightarrow \mathbb{R}^{2} \setminus \{ 0 \} $ smooth. Define

\begin{equation*}
    \mathrm{Wind}(\gamma,0)=\frac{1}{2\pi} \int_{C} \gamma^{*}(d\theta),
\qquad d\theta=\frac{x\,dy-y\,dx}{x^{2} + y^{2}}.
\end{equation*}

In the holomorphic setting, winding numbers are related to multiplicities of zeroes via the argument principle. A point $ (b,c) \in  S $ is a zero of $F$ if $ F(b,c)=0 $.
Assume $ (b,c) $ is an isolated zero, that is, $ F\neq 0 $ on $ \gamma_{\epsilon}(b,c) = \partial B_{\epsilon}(b,c) $ for some sufficiently small $ \epsilon > 0 $.
Then we define its multiplicity by

\begin{equation*}
   \mathrm{mult}(b,c):=\mathrm{Wind}(F|_{\gamma_{\epsilon}(b,c)},0).
\end{equation*}

In \cite[Theorem~7.4]{joyce_special_2002}, Joyce proves that the multiplicity of any isolated
zero $ (b,c) $ of $ (u_{1},v_{1}) - (u_{2},v_{2}) $ in $ S^{\circ} $ is a positive integer. We expect that the winding number results in \cite{joyce2005u1, joyce2005u2, joyce2005u3} extends to this paper's structure with no significant change.

\bibliographystyle{plainnat} 
\bibliography{references}    

\end{document}